\documentclass[12pt]{elsarticle}
\usepackage{latexsym}
\usepackage[centertags]{amsmath}
\usepackage{amsfonts}
\usepackage{amssymb}
\usepackage{amsthm}
\usepackage{newlfont}
\usepackage{graphics}
\usepackage{color}

\usepackage{hyperref}
\usepackage{bookmark}

\newtheorem{thm}{Theorem}%[section]

\newtheorem{lem}[thm]{Lemma}
\newtheorem{prop}[thm]{Proposition}

\newtheorem{exa}[thm]{Example}

\newtheorem{prob}{Problem}

\newtheorem{rem}[thm]{Remark}

\def\CT{\mathop{\mathrm{CT}}}
\allowdisplaybreaks[1]

\def\srem{\;\mathrm{srem}}
\def\mrem{\;\mathrm{rem}}

\title{A Euclid style algorithm for MacMahon's partition analysis}

\author{Guoce Xin $^{1,2}$\\
{\small $^{1}$ School of Mathematical Sciences, Capital Normal University, Beijing 100048, PR China}\\%[-0.8ex]
{ $^{2}$ Beijing Center for Mathematics and \\
Information Interdisciplinary Sciences, Beijing, 100048, PR China}\\
\small \texttt{guoce.xin@gmail.com}\\
}
\date{Oct. 20, 2012}
\begin{document}

\maketitle

%\begin{small}
%\noindent
%Abstract:
%Solutions to a linear Diophantine system, or lattice points in a
%rational convex polytope, are important concepts in algebraic
%combinatorics and computational geometry. The enumeration problem is
%fundamental and has been well studied, because it has many
%applications in various fields of mathematics. In algebraic
%combinatorics, MacMahon partition analysis has become a general
%approach for linear Diophantine system related problems. Many
%algorithms have been developed, but ``bottlenecks" always
%arise when dealing with complex problems. While in computational
%geometry, Barvinok's  important result asserts the existence of a polynomial time algorithm
%when the dimension is fixed. However, the implementation by the LattE
%package of De Loera et. al. does not perform well in many situations. By
%combining excellent ideas in the two fields, we generalize Barvinok's
%result by giving a polynomial time algorithm for MacMahon partition
%analysis in a suitable condition. We also present an elementary Euclid style algorithm,
%which might not be polynomial but is easy to implement and performs well.
%As applications, we contribute the generating series for magic squares of order 6.
%\end{small}

\begin{abstract}
Solutions to a linear Diophantine system, or lattice points in a
rational convex polytope, are important concepts in algebraic
combinatorics and computational geometry. The enumeration problem is
fundamental and has been well studied, because it has many
applications in various fields of mathematics. In algebraic
combinatorics, MacMahon's partition analysis has become a general
approach for linear Diophantine system related problems. Many
algorithms have been developed, but ``bottlenecks" always
arise when dealing with complex problems. While in computational
geometry, Barvinok's  important result asserts the existence of a polynomial time algorithm
when the dimension is fixed. However, the implementation by the LattE
package of De Loera et. al. does not perform well in many situations. By
combining excellent ideas in the two fields, we generalize Barvinok's
result by giving a polynomial time algorithm for MacMahon's partition
analysis in a suitable condition. We also present an elementary Euclid style algorithm,
which might not be polynomial but is easy to implement and performs well.
As applications, we contribute the generating series for magic squares of order 6.

\end{abstract}

\medskip



\section{Introduction}

The linear Diophantine system is one of the most fundamental concepts in mathematics. One basic problem is to
determine the set of non-negative integer solutions of a system of linear equations (or inequalities)
$A\alpha=b$ for a suitable integral matrix $A$ and vector $b$. In the context of geometry, the problem is to determine
the lattice points in a rational convex polyhedron $P=\{\alpha : A\alpha =b,\; \alpha \ge 0 \}$ specified by $A$ and $b$. That is, we need to determine
the set $P\cap \mathbb{Z}^n$.
If $b=0$ then the linear Diophantine system is called homogeneous, and the corresponding $P$ is called a rational cone.
This basic problem received much attention for its wide application in many fields of mathematics. Many theories have been developed but this paper will address the algorithmic aspect. There are many algorithms dealing with linear Diophantine system related problems. Two algorithms in two different fields have great advantages (in running time). One is Barvinok's polynomial time algorithm in computational geometry \cite{Barvinok} and the other is the author's partial fraction algorithm \cite{iterate} in algebraic combinatorics.

We will use the shorthand notation $x^\alpha=x_1^{\alpha_1} x_2^{\alpha_2} \cdots x_n^{\alpha_n}$ throughout this paper.
To understand the complexity of this problem, let us specify that $A$ is an $r\times n$ matrix of rank $r$
and consider the homogeneous case. The structure of the $\mathbb{Z}$-solutions $\{ \alpha \in \mathbb{Z}^n: A\alpha =0 \}$ is simple, since they
form a subgroup of $\mathbb{Z}^n$ with $n-r$ generators, which can be obtained through Hermite normal form. The structure of nonnegative integer solutions
$E=\{ \alpha \in \mathbb{N}^n: A\alpha =0 \}$ is only a free commutative monoid (semigroup with identity). There is no simple way to enumerate the elements of $E$, which is equivalent to the construction of the \emph{rational} generating series
$$ E(x)=E(x;A,0)= \sum_{\alpha \in E} x^{\alpha}=\frac{P(x)}{(1-x^{\beta_1})(1-x^{\beta_2})\cdots (1-\beta^{\beta_N})},$$
where $P(x)$ might be a monster polynomial and $\{\beta_i\}_{1\le i\le N}$ consists of all complete fundamental elements of $E$ \cite[Thm.~4.6.11]{EC1}. See \cite[Ch.~6]{EC1} for related terminology and combinatorial theories developed there. Many practical problems can be solved by some specializations of ${E}(x)$, such as ${E}(q,q,\dots, q)$.
It is worth mentioning that there is a beautiful reciprocity theorem for rational cones due to Stanley, which gives a simple connection between nonnegative solutions and positive solutions.

Our algorithms are under the framework of algebraic combinatorics, but borrow some beautiful ideas from computational geometry. The core problem is to compute the constant term in $\Lambda=(\lambda_1,\dots, \lambda_r)$ of an Elliott-rational function $\mathcal E$, written as
\begin{align}
\CT_\Lambda \mathcal E=   \CT_{\Lambda} \frac{L}{(1-M_1)(1-M_2)\cdots (1-M_n)},\label{e-Introduction}
\end{align}
where $L$ is a Laurent polynomial in $\Lambda$ and $ M_i$ is a monomial containing the $\lambda$ variables for each $i$. Following \cite{iterate}, here we specify a field of iterated Laurent series, called the working field, to clarify the series expansion of rational functions. George Andrews found that MacMahon's partition analysis could be ideally combined with computer algebra for dealing with linear Diophantine system related problems \cite{MPA1}. Andrews and his coauthors have published a series of 12 papers on this topic. The first such algorithm was developed by Andrews et al. and implemented by the Mathematica package \texttt{Omega} \cite{MPAA1}. An improvement appears in \cite{MPAA2}. These algorithms rely on the \emph{unique series expansion} of rational functions.
The author made significant progress with his partial fraction algorithm \cite{iterate}, which is  implemented in the updated Maple package \texttt{Ell2}. The practical running time of \texttt{Ell2} is much faster than that of \texttt{Omega}. The author's partial fraction algorithm made two major contributions: 1) It introduced the framework of iterated Laurent series to guarantee the unique series expansion of all rational functions. Note that rational functions like $(\lambda_1-\lambda_2)^{-1}$ were not allowed to appear in the framework of complex analysis. 2) It gave an efficient formula for computing the partial fraction decompositions, and the constant term can be read off separately. However, there are bottlenecks for this approach when the denominator has multiple roots or nonlinear factors. See Section \ref{s-Euclid} for a detailed description. We find the ideas from computational geometry helpful in resolving these bottlenecks.

A simpler model has been extensively studied earlier in computational geometry. Let $P$ be as above specified by integral matrix $A=(a_{ij})_{r\times n}$ and nonzero vector $b$. Then it is well-known that the generating function for $P\cap \mathbb{N}^n$ can be written as a constant term:
\begin{align}
E(x;A,b)=\sum_{\alpha \in P\cap \mathbb{N}^n} x_1^{\alpha_1} \cdots x_n^{\alpha_n} = \CT_{\Lambda} \frac{\lambda_1^{-b_1} \cdots \lambda_r^{-b_r}}{
\prod_{j=1}^n (1- \lambda_1^{a_{1,j}} \lambda_2^{a_{2,j}} \cdots \lambda_r^{a_{r,j}} x_j)}. \label{e-Introduction-P}
 \end{align}
So this is just the special case of \eqref{e-Introduction} when the numerator $L$ is a monomial. If $P$ is bounded, then $P$ is called a polytope and $E(x;A,b)$ is a polynomial. Geometers are mainly interested in the specialization $x_i=1$ for all $i$, which gives the number of lattice points in $P$. The most important result in this field is due to Barvinok, who developed a polynomial time algorithm when the dimension is fixed \cite{Barvinok} in 1994. Barvinok's algorithm relies on solving linear programming problems. It was implemented by the \texttt{LattE} package by De Loera et al. \cite{LattE} in 2004. An improvement was given in \cite{LattEup}. The readers are referred to \cite{LattE} for related references and \cite{Beck} for related applications. The author's algorithm is not polynomial time, but the \texttt{Ell2} package has better performance than the \texttt{LattE} package when $r$ and the entries of $(A,b)$ are small. The two algorithms are very different in nature, as described in Section 2. We find that we can do better if we combine the beautiful ideas of the two algorithms.

This paper is organized as follows. Section 1 is this introduction. Our ultimate goal is to develop a classic algorithm for this subject in the near future. Section 2 introduces and compares the ideas of Barvinok's polynomial time algorithm  and the author's partial fraction algorithm. Section 3 includes one of the two major contributions in this paper: We extend Barvinok's algorithm for the multivariate specialization, which gives rise to a polynomial time algorithm for MacMahon's partition analysis. We do not give an implementation of this algorithm since it involves too much geometry and there is much room for improvement. Section 4 includes the other major contribution. We develop a Euclid style algorithm with an implementation by the Maple package \texttt{CTEuclid}. This algorithm performs well and resolves several bottlenecks in the \texttt{Ell2} package. In Section 5, we give an introduction for \texttt{CTEuclid} with concrete examples for the sake of clarity. We explain the flexibility of our algorithm and our strategy for benchmark problems. As an application, we give the first solution for the generating function of order 6 magic squares.

\section{Comparison of Barvinok's polynomial time algorithm and the author's partial fraction algorithm}
Barvinok's algorithm is in the context of computational geometry and the author's algorithm is along the line of MacMahon's partition analysis in algebraic combinatorics. In this section we compare the two algorithms and conclude that a better strategy is to combine the nice ideas of the two algorithms.
It would be helpful to give a brief description of the two algorithms. We follow the notation from the introduction and make the following clarification:

Throughout this paper, $d$ is always referred to as the dimension or an upper bound for the dimension of the corresponding polytope. We always assume that the matrix $A_{r\times n}$ has rank $r$. Then the dimension of the null-space of $A$ is $d=n-r$.
The polytope $P$ specified by $A$ and $b$ is the intersection of $\mathbb{R}^n$ with a certain shift of the null-space of $A$, so the (affine) dimension of $P$ is no more than $d$. It is possible that $P$ has a lower dimension. For example, if $(A,b)=(1,1,-1)$ then $P$ is empty. In general it is not easy to find the exact dimension of $P$, but the bound $d$ is sufficient for our purpose.

\medskip
\noindent
Barvinok's Algorithm for computing $E(x;A,b)\big|_{x_i=1}$, where $(A,b)$ specifies a bounded polytope $P$ with fixed dimension $d$. 

\begin{enumerate}
\item By using Brion's theorem, we can write
  $$ E(x;A,b)= \sum_i  x^{\nu^{(i)}} E(x; A^{(i)},0),$$
  where the summands correspond to vertex cones and the sum ranges over all vertices $\nu^{(i)}$ of $P$.

\item A fixed dimensional rational cone can be signed-decomposed into simplicial cones (see e.g. \cite{Aurenhammer,Lee}).

\item Barvinok made the \emph{key observation} that a simplicial cone can be decomposed into polynomially many unimodular cones.

\item The generating function for a unimodular cone has a monomial numerator. Thus $E(x;A,b)$ can be written as a sum of polynomially many simple rational functions, for which we refer to
Elliott-rational functions with monomial numerators.

\item Finally, take limits when $x_i \to 1$ for all $i$, as we shall discuss in Section \ref{s-slackvariables}.
\end{enumerate}

The above outlined algorithm follows that of \cite{LattE}. Note that in geometry, a polytope is usually defined by inequalities (or together with some equalities), which is inconsistent with our definition by equalities. This problem can be easily solved for a simplex, and for the general situation Cook et al. \cite{Cook} and Dyer \cite{Dyer} showed that the problem of counting integral points in a rational polytope can be reduced in polynomial time to counting integral points in
an integral simplex assuming that the dimension is fixed.

\medskip
\noindent
The XinPF Algorithm computes the constant term of the Elliott rational function $\mathcal E$ as in \eqref{e-Introduction} regarded as an element in a specified field $K$ of iterated Laurent series.
\begin{enumerate}
\item Take $O_0$ as a sum with the single term $E$, and compute $O_{i+1}$ from $O_{i}$ for $i=0,1,\dots, r-1$ as follows. Then $O_{r}$ will be the output.

      \item For each summand $O_{ij}$ in $O_i$, choose a variable $\lambda$, compute $\CT_\lambda O_{ij}$ as in the next step, and collect the results into $O_{i+1}$.

      \item When computing $\CT_{\lambda} O_{ij}$, we first compute the partial fraction decomposition of $O_{ij}$ with respect to $\lambda$ and then read off the constant term (see Section \ref{s-XinPF}). The result is a sum of Elliott-rational functions in $K$, and we do not combine the sum to a single rational function.
\end{enumerate}

Let us see a simple example. Suppose we want to compute the number $N$ of lattice points in the $5\times 5$ square in the plane
with vertices $(0,0),(5,0),(0,5),(5,5)$. From the geometric side, Brion's theorem gives 4 pointed unimodular cones directly, and hence the sum of 4 simple rational
functions
$$\frac{1}{(1-x)(1-y)} + \frac{x^5}{(1-x^{-1})(1-y)} + \frac{y^5}{(1-x)(1-y^{-1})} +\frac{x^5y^5}{(1-x^{-1})(1-y^{-1})}. $$
For instance, the second term corresponds to the vertex cone with vertex $(5,0)$ (thus with numerator $x^5$), and generators $(-1,0)$ and $(0,1)$ (corresponding to $x^{-1}$ and $y$ in the denominator).
Now taking limit as $x,y\to 1$ gives $N=36.$

If we use partition analysis, the system is given by $\{a_1\le 5, a_2\le 5, a_i\ge 0\}$, which should be transformed to the equalities
$\{a_1+a_3 =5, a_2+a_4=5, a_i\ge 0\}$ with added variables $a_3,a_4$. Thus we have
\begin{align*}
  N = \sum_{a_1+a_3=5,a_2+a_4=5,a_i\ge 0} 1 &=\sum_{a_i\ge 0} \CT_{\lambda_1,\lambda_2} \lambda_1^{a_1+a_3-5} \lambda_2^{a_2+a_4-5} \\
    &=\CT_{\lambda_1,\lambda_2} \frac{\lambda_1^{-5}\lambda_2^{-5}}{(1-\lambda_1)^2(1-\lambda_2)^2} \\
    &=\CT_{\lambda_1} \frac{\lambda_1^{-5}}{(1-\lambda_1)^2} \times \CT_{\lambda_2} \frac{\lambda_2^{-5}}{(1-\lambda_2)^2} =36.
\end{align*}

\medskip
\noindent
\textbf{The comparison}

\medskip
From the above description, we see that the two algorithms are very different.
\begin{enumerate}
      \item The basic elements of Barvinok's algorithm are rational cones while that of the XinPF algorithm are Elliott rational functions, which include a larger class of objects.

      \item By performing computations on rational cones, Barvinok's algorithm avoids the convergence problem. The XinPF algorithm settles the convergence problem by introducing the field of iterated Laurent series as a framework.

\item Barvinok's algorithm is polynomial while the XinPF algorithm is not.

\item The XinPF algorithm can handle polynomial numerators while Barvinok' algorithm can only handle monomial numerators.

\item The application of partial fraction decompositions has much more flexibility than rational cone decompositions under our framework.
\end{enumerate}

Although Barvinok's algorithm is polynomial, the application of Brion's theorem may be very costly if the number of vertices is large. Indeed, the author's \texttt{Ell2} package has better performance than the \texttt{LattE} package in many situations, such as the Sdd5 problem as we will discuss in Section \ref{s-sdd}. Because of the freedom of the author's algorithm,
we conclude that a better strategy is to embed some of Barvinok's ideas into the author's framework. This leads to a polynomial time algorithm for MacMahon's partition analysis in Section 3.

\section{A polynomial time algorithm for MacMahon's partition analysis in theory}
Our objective in this section is to give a polynomial time algorithm for the following core problem in MacMahon's partition analysis.

\subsection{The core problem and the polynomial time algorithm}
We need some notation and assumptions. Let $U$ be a fixed positive integer. Denote by $K$ the field of iterated Laurent series
specified by
the sequence of variables $(x_1,x_2,\dots,x_m)$. See Section \ref{s-XinPF} for a brief introduction. Here just keep in mind that every monomial $M\ne 1$ satisfies either $M<1$ or $M>1$. We will frequently rewrite our rational functions using the fact that
$1/(1-M)= -M^{-1}/(1-M^{-1})$.

i) An Elliott rational function $E=E(x_1,x_2,\dots, x_m)$ will be written as
\begin{align}
  \label{e-E-heart}
E = \frac{L}{(1-M_1(x)\Lambda^{c_1})(1-M_2(x)\Lambda^{c_2})\cdots (1-M_n(x) \Lambda^{c_n})},
\end{align}
where $L$ is a linear sum of $U$ monomials in $x$ and each $M_i(x)$ is a monomial in $x$ but free of $\Lambda$.
Moreover, we require $M_i(x)\Lambda^{c_i}<1$ in $K$ for all $i$.

ii) The integer matrix $A_{r\times n}=(c_1,\dots, c_n)$ has full rank $r$ ($\le n$).

iii) We fix $d=n-r$, and call it the dimension of the core problem.

\begin{prob}[Core Problem]\label{prob-heart}
\emph{Given} an Elliott rational function $E$ as in \eqref{e-E-heart},
  and a set $\Lambda=\{\lambda_1,\dots, \lambda_r \}\subseteq \{x_1,x_2,\dots,x_m \}$ of variables to be eliminated, \emph{represent} the constant term in $\Lambda$ of $E$ when working in $K$ as follows.
$$ \CT_{\lambda_1,\dots, \lambda_r} E= \text{ a short sum of simple rational functions free of the $\lambda$'s}. $$
Here a \emph{short} sum is a sum with polynomially many terms, and a \emph{simple} rational function refers to rational functions with no more than $C(d)U$ monomials in their numerators, where $C(d)$ is a constant depending only on $d$.
\end{prob}

Note that there is no convergence problem: by the structure of $K$, the constant term of $E$ is still an iterated Laurent series.

\begin{thm}
For a fixed number $U$ and dimension $d$, the core problem can be solved in polynomial time.% in the sense of Barvinok's.
\end{thm}
\begin{proof}
By linearity, it is sufficient to consider the $U=1$ case. Assume $L=\Lambda^{-b}$.

It is clear from direct series expansion that
$$ \CT_\Lambda E= \sum_{A\alpha =b, \alpha\in \mathbb{N}^n} M_1(x)^{\alpha_1} M_2(x)^{\alpha_2}\cdots M_n(x)^{\alpha_n}.$$
This is essentially the computation of
$E(y;A,b)\big|_{y_i=M_i(x),i=1,2,\dots, n}$ for the polyhedron $P$ specified by $A$ and $b$.

If $P$ is bounded, then by Barvinok's Algorithm, but without the last step, we have a short sum representation of $E$ in the $y$'s.

For general $P$, we can also obtain a short sum representation of $E$ in the $y$'s.
We can introduce a new variable $q$ and use the formula
$$ \CT_\Lambda E= [q^1] \; \CT_\Lambda \frac{1}{(1-M_1(x)\Lambda^{c_1})(1-M_2(x)\Lambda^{c_2})\cdots (1-M_n(x) \Lambda^{c_n}) (1-q \Lambda^{-b})} .$$
Here we work in the field $K((q))$ of Laurent series in $q$ with coefficients in $K$.
For a similar reason, this corresponds to computing $[q^1] E(y,q;(A,-b))\big|_{y_i=M_i(x),i=1,2,\dots, n}$ for the rational cone specified by $(A,-b)$.
Still by  Barvinok's Algorithm, steps 2-4, we have a short sum representation in the $y$'s and $q$. Since each summand corresponds to a unimodular cone of dimension no more than $d+1$, it can be written, up to a constant scalar (in $q$), in the following form.
$$T =\frac{1}{(1-y^{\beta_1}q^{a_1}) \cdots (1-y^{\beta_k} q^{a_k}) (1-y^{\beta_{k+1}} q^{-a_{k+1}}) \cdots (1-y^{\beta_{k'}} q^{-a_{k'}}) },$$
where $a_i$ are all positive integers and $k'\le d+1$. In $K((q))$ it is easy to obtain:
\begin{align*}
[q^1]\;  T=\left\{
                                                                                                                                 \begin{array}{ll}
                                                                                                                                   0, & \text{ if }a_{k+1}+\cdots +a_{k'}>1 \\
                                                                                                                                  -y^{-\beta_{k+1}}, & \text{ if } k'=k+1 \ \& \ a_{k+1}=1 \\
\sum_{i} y^{\beta_i} \chi(a_i=1), & \text{ otherwise } k'=k,
\end{array}
\right.
\end{align*}
where $\chi(S)$ is 1 if the statement $S$ is true and $0$ if otherwise.

Applying the above formula to each summand gives a short sum representation of $E$ in the $y$'s.

Setting
$y_i=M_i(x) z_i$ with $z_i$'s the slack variables, we are left with taking the limit when $z_i\to 1$ for all $i$.
The polynomial time result will follow if computing the limit takes polynomial time. This will be shown in Section \ref{s-slackvariables}.
\end{proof}
The number $U$ plays no role in our proof, but it is significant in practice, as will be seen later in Remark \ref{rem-slack2}.

\subsection{Dispelling the slack variables \label{s-slackvariables}}
The algorithm for dispelling the slack variables $z_i$ for all $i$ could be thought of  as an independent subject. Although the idea works for a more general class of rational functions, we concentrate on the class of Elliott-rational functions that arise naturally from MacMahon's partition analysis.

\begin{prob}\label{prob-slack}
\emph{Given} an Elliott-rational function $Q(x_1,\dots, x_m;z_1,\dots,z_n)$ written as a short sum:
$$ Q= \sum_i  \frac{L_i(x;z)}{(1-M_{i1}z^{B_{i1}})(1-M_{i2}z^{B_{i2}})\cdots (1-M_{id}z^{B_{id}})},$$
where the $M_{ij}$'s are monomials in $x$, the $L_i$'s are Laurent polynomials with at most $U$ monomials, \emph{represent} $Q(x_1,\dots, x_m;1,\dots,1)$,
%=\lim_{z_i=1,i=1,\dot,n}Q$,
which is known in advance to be well defined, as a short sum of simple rational functions:
$$Q(x_1,\dots, x_m;1,\dots,1)= \sum \frac{\text{linear combination of $UC(d)$ monomials}}{\text{product of binomials}}. $$
\end{prob}

It is not clear how to dispel $z_i$ even when $m=0$: i) combining the terms to a single rational function will get a monster numerator that can not be handled by the computer; ii) direct substitution of $z_i=1$ for all $i$ does not work for possible denominator factors like $1-z_1z_2$ in some of the terms.

The algorithm we present here is inspired by the idea from computational geometry, which is
better illustrated by the following simple example. We have
$$ \lim_{z\to 1} \frac{1-z^n}{1-z} = \lim_{z\to 1} \sum_{i=0}^{n-1} z^i =\sum_{i=0}^{n-1} 1 =n.$$
 Computationally this is inefficient for large values of $n$. Of course there are many other methods, but the following computation extends for Problem \ref{prob-slack}, provided that we know the existence of the limit. We have (by making the substitution $z\to 1+s$)
\begin{align*}
  \lim_{z\to 1} \frac{1}{1-z}-\frac{z^n}{1-z} &= \lim_{s\to 0} \frac{1}{-s} + \frac{(1+s)^n}{s} = \CT_s \frac{1}{-s} + \CT_s \frac{(1+s)^n}{s} =n.
\end{align*}

\medskip
\noindent
The DispelSlack Algorithm for Problem \ref{prob-slack} consists of two major steps.

\begin{enumerate}
  \item \emph{Reduce the number of slack variables to $1$}. Calculating $Q(x;1,1,\dots, 1)$ is equivalent to evaluating the limit as $z_i$ goes to $1$ for all $i$. Our first step is to reduce the number of slack variables to $1$. This is done by finding a suitable integer vector $\lambda$ and making the substitution $z_i\to t^{\lambda_i}$. In order to do so, $\lambda$ must be picked such that there is no zero denominator in any term, i.e., for every $i$ and $j$ we can not have both $M_{ij}=1$ and the inner product $\langle \lambda, B_{ij}\rangle=0$. Barvinok showed such $\lambda$ can be picked in polynomial time by choosing points on the moment curve. De Loera et al. \cite{LattE} suggested using random vectors to avoid large integer entries.

\item  \emph{Use the Laurent series expansion}. Now we need $Q(x; t^{\lambda_1},\dots, t^{\lambda_n})\Big|_{t=1}$, where
$$Q(x;t^{\lambda_1},\dots, t^{\lambda_n})= \sum_i \frac{L_i(x; t^{\lambda_1},\dots, t^{\lambda_n})}{\prod (1-t^{\langle\lambda, B_{ij}\rangle}M_{ij})}.$$

An obvious way is to make the substitution $t=1+s$. Then we have
$$ Q%(x;t^{\lambda_1},\dots, t^{\lambda_n})
\Big|_{t=1}
=Q(x;(1+s)^{\lambda_1},\dots, (1+s)^{\lambda_n})\Big|_{s=0}=\CT_s Q(x;(1+s)^{\lambda_1},\dots, (1+s)^{\lambda_n}),$$
where we are taking the constant term of a Laurent series in $s$. The linearity of the constant term operator allows us to compute separately:
$$Q(x;1,\dots, 1)= \sum_i  \CT_s \frac{L_i(x; (1+s)^{\lambda_1},\dots, (1+s)^{\lambda_n})}{\prod (1-(1+s)^{\langle\lambda, B_{ij}\rangle}M_{ij})}.$$

The substitution $t=1+s$ seems natural and works fine for the $m=0$ case, where we only need do polynomial division in $\mathbb{Q}[s]$. For $m\ge 1$, we find it better to make the exponential substitution $t=e^s$, which leads to
$$Q(x;1,\dots, 1)= \sum_i  \CT_s \frac{L_i(x; e^{\lambda_1s},\dots, e^{\lambda_n s})}{\prod (1-e^{s\langle\lambda, B_{ij}\rangle}M_{ij})}.$$
Now applying the following proposition to each summand gives the desired result.
\end{enumerate}

\begin{prop}\label{p-slacksingle}
 Let $L(x;z)$ be a Laurent polynomial with $U$ monomials, $M_j$ monomials in $x$ and $b_j$ integers for $1\le j \le d$. The constant term
$$\CT_s \frac{L(x; e^{\lambda_1s},\dots, e^{\lambda_n s})}{\prod_{j=1}^d (1-e^{b_js}M_{j})}$$
can be efficiently computed as a sum of at most $\binom{d+1}{\lceil d/2\rceil}$ rational functions, where each has at most $U C(d)$ monomials in the numerator.
\end{prop}
For fixed $d$, we only need the following three formulas and do polynomial multiplications if using the exponential substitution:
\begin{align}
  e^s &= \sum_{n=0}^d \frac{1}{n!} s^n +o(s^d) \label{e-exp}\\
  \frac{s}{1-e^{s}} &= \sum_{n=0}^d- \frac{\mathcal B_n}{n!} s^n+o(s^{d})=-1+{\frac {1}{2}}s-{\frac {1}{12}}{s}^{2}+{\frac {1}{720}}{s}^{4}
  %-{\frac {1}{30240}}{s}^{6}
+\cdots+o(s^d), \label{e-bernulli}\\
\frac{1}{1-e^{s}M}  &= \frac{(1-M)^{-1}}{ (1- \frac{M}{1-M}(e^{s}-1))}
= \sum_{n\ge 0 } \frac{M^n(e^{s}-1)^n}{(1-M)^{n+1}} = \sum_{n= 0}^d c_n(M) s^n +o(s^d). \label{e-esM}
\end{align}
The $\mathcal B_n$ are the well-known Bernoulli numbers and $c_n(M)$ has denominator $(1-M)^{n+1}$. When programming, the $\mathcal{B}_n$ and $c_n(M)$ can be stored in advance for all $n\le d$.

\begin{proof}[Proof of Proposition \ref{p-slacksingle}]
By rearranging the factors of the denominator, we may assume that $M_1=M_2=\cdots =M_r=1$ and all the other $M_j$ are not $1$ for some $r$.
Now by \eqref{e-bernulli} we can use $r-1$ multiplications to obtain
$$ \frac{s^r}{\prod_{j=1}^r (1-e^{b_js}M_{j})} =\prod_{j=1}^r \sum_{n\ge 0}- \frac{\mathcal B_nb_j^{n-1}}{n!} s^n= \sum_{n=0}^r c'_n s^n + o(s^r).$$
This is a power series in $s$.
It follows that
$$ \CT_s \frac{L(x; e^{\lambda_1s},\dots, e^{\lambda_n s})}{\prod_{j=1}^d (1-e^{b_js}M_{j})} =[s^r] \frac{L(x; e^{\lambda_1s},\dots, e^{\lambda_n s})}{\prod_{j=r+1}^d (1-e^{b_js}M_{j})}\cdot \sum_{n=0}^r c'_n s^n. $$
So we are indeed taking the coefficient of $s^r$ in a product of power series.

By using \eqref{e-exp}, we clearly have the expansion
$$L(x; e^{\lambda_1s},\dots, e^{\lambda_n s})= \sum_{n_0=0}^r \ell_{n_0}(x) s^{n_0} +o(s^r),$$
and $\ell_{n_0}$ has at most $U$ monomials. Therefore we have, by \eqref{e-esM},
\begin{align*}
  s^r \frac{L(x; e^{\lambda_1s},\dots, e^{\lambda_n s})}{\prod_{j=1}^d (1-e^{b_js}M_{j})}
&= \sum_{n_0= 0}^r \ell_{n_0}(x) s^{n_0} \cdot \sum_{n_1= 0}^r c'_{n_1} s^{n_1} \prod_{j=r+1}^d \sum_{n_j= 0}^r c_{n_j}(M_j)b_j^{n_j} +o(s^r)\\
&= \sum_{n= 0}^r s^n \sum_{n_0+n_1+n_{r+1}+\cdots +n_d =n} \ell_{n_0}(x) c_{n_1}' \prod_{j=r+1}^d c_{n_j}(M_j)b_j^{n_j} +o(s^r).
\end{align*}

It follows that
the desired constant term is given by
\begin{align}
  \label{e-slack-monster}
  \sum_{n_0+n_1+n_{r+1}+\cdots +n_d =r}\ell_{n_0}(x) c_{n_1}' \prod_{j=r+1}^d c_{n_j}(M_j)b_j^n.
\end{align}
The number of terms is equal to the number of nonnegative integer solutions of $n_0+n_1+n_{r+1}+\cdots +n_d =r$, which is easily shown to be
$\binom{d+1}{r}\le \binom{d+1}{\lceil d/2 \rceil}$.

Finally for any single term $\ell_{n_0}(x) c_{n_1}' \prod_{j=r+1}^d c_{n_j}(M_j)b_j^n$, the denominator is a product of binomials and the numerator has at most
$U C(d)$ monomials with
$$C(d) \le (\max_{j\le d} \{\text{ the number of monomials in } c_j(M)\})^d.$$
The proposition then follows.
\end{proof}

\begin{rem}\label{rem-slack1}
In practice, we simply compute the constant term in Proposition \ref{p-slacksingle} by $d$ (or $d-1$ if $U=1$) multiplications of elements
in $\mathbb{Q}(x_1,\dots,x_m)[s]$ with degree no more than $r$. This gives the same complexity since the elements we multiply are of fixed type coming from formulas $($\ref{e-bernulli},\ref{e-esM}$)$. This is the first choice when $m\le 1$ since the size of the numerator has a natural bound by degree. But for $m>1$ the size of the numerator maybe too large for computational purposes. We do not develop an algorithm using \eqref{e-slack-monster}, because we have not found significant advantages for the $m\ge 2$ cases.
The bound $C(d)$ and \eqref{e-slack-monster} are only used for complexity.
%in the theory of polynomial time algorithm.
\end{rem}

\begin{rem}\label{rem-slackp}
Let $p>d$ be a prime number that does not divide the denominator of each coefficient of $Q$ (or $E$ in Problem \ref{prob-heart}). Then we can compute $Q(x_1,\dots, x_m;1,\dots,1) \pmod{p}$ by performing modular arithmetic with respect to $p$ at each step. The only formulas we need to modify are
(\ref{e-exp}--\ref{e-esM}).
\end{rem}

\begin{rem}\label{rem-slack2}
The running time $R(U)$ of DispelSlack does not vary much when the number $U$ changes, especially when $m=0,1$. As described in Remark \ref{rem-slack1}, the ratio $R(U)/R(1)$ is approximately $d/(d-1)$ unless $U$ is too large. The ratio is about the same for the general algorithm for MacMahon's partition analysis, since the DispelSlack step takes most of the running time.
%This means that when computing the constant term of an Elliott-rational function with (for instance) $U=1000$  monomials in its numerator, it is not a good idea to split the problem into 1000 similar subproblems with monomial numerator. An algorithm dealing with polynomial numerators is more desirable in practice.
\end{rem}
From Remark \ref{rem-slack2}, we conclude that a good algorithm for Problem \ref{prob-heart} should be able to handle polynomial numerators ``uniformly": If we can only deal with monomial numerators, then the running time $R(U)$ will be roughly $U\cdot R(1)$ since we have to split the problem into $U$ similar problems. This is crucial in practice when $U$ is large, say 100. This is one of the reasons why we only claim a polynomial time algorithm in theory but without an implementation.

\begin{rem}
In Problem \ref{prob-slack} we can also detect (when not given) the existence of\\
 $Q(x_1,\dots, x_m,1,\dots,1)$ without affecting the complexity of the problem. This is done by using the same idea to verify
the equalities
$$[s^{k}] Q(x;e^{\lambda_1 s},\dots, e^{\lambda_n s})=0, \text{ \emph{for} } k=-d, -d+1,\dots, -1.$$
\end{rem}

The only problem for this approach is the large integer problem. This is crucial in practice but do not affect the complexity for fixed dimension $d$. It seems unavoidable that some of the $b_{ij}=\langle \lambda, B_{ij}\rangle$ might be large. This results in huge numbers (but polynomial in the input) since our formula involves $b_{ij}^n$ for $n\le d$. This problem is avoidable by choosing a reasonably large prime number $p$ and using Remark \ref{rem-slackp}. Then we can construct the final output by some other known information. This is indeed the case since our problems are usually combinatorial and the final output is nice in some sense. The $m=0$ case of the problem usually computes the number of lattice points in a rational convex polytope. There are methods to estimate this number. For the $m=1$ case we usually compute the Ehrhart series as in Section \ref{s-Ehrhart}. Such generating functions always have known denominators and their numerators have integer coefficients which are not too large. Thus we can do the constant term extraction modulo $p$. If necessary, we can do the computation several times using different large primes and then use the Chinese remainder theorem to construct the final output.

For the above reasons, it is necessary to make progress on an  algebra-based algorithm which can handle polynomial numerators. This is why we develop the Euclid style algorithm in Section 4.

\section{A Euclid style algorithm for MacMahon's partition analysis \label{s-Euclid}}
The Euclid style algorithm is developed in the framework of MacMahon's partition analysis. It is an elimination-based algorithm, so the basic problem is to take constant terms in a single variable. For this problem, we provide an elementary approach like the Euclidean algorithm for greatest common divisors. With the help of some geometric ideas, we are able to resolve several bottlenecks in the \texttt{Ell2} package.

\subsection{Brief introduction to the XinPF Algorithm\label{s-XinPF}}
The Euclid style algorithm is along the line of the author's partial fraction algorithm, so it is time to explain briefly the field of iterated Laurent series. We use the list $\texttt{vars}=[x_1,x_2,\dots, x_m]$ to define the working field $K=\mathbb{Q}((x_m))((x_{m-1}))\cdots ((x_1))$. See \cite{iterate} for a detailed explanation. Here we only need the fact that every monomial $M\ne 1$ is comparable with $1$ in $K$ by the following rule:
find the ``smallest" variable $x_j$ appearing in $M$, i.e, $\deg_{x_i} M=0$ for all $i<j$. If $\deg_{x_j} M>0$ then we say $M$ is small, denoted $M<1$, otherwise we say $M$ is large, denoted $M>1$. Thus we can determine which of the following two series expansion holds in $K$.
\begin{align*}
  \frac{1}{1-M} =\left\{
                   \begin{array}{ll}
                    \displaystyle \sum_{k\ge 0} M^k, & \text{ if } M<1; \\
                    \displaystyle \frac{1}{-M(1-1/M)}=  \sum_{k\ge 0} - \frac{1}{M^{k+1}}, & \text{ if } M>1.
                   \end{array}
                 \right.
\end{align*}
When expanding $E$ as a series in $K$, we usually rewrite $E$ in its \emph{proper form}:
\begin{align}\label{e-properform}
  E = \frac{L}{(1-M_1)(1-M_2)\cdots (1-M_n)},  \tag{proper form}
\end{align}
where $L$ is a Laurent polynomial and $M_i<1$ for all $i$. Note that the proper form of $E$ is not unique. For instance
$1/(1-x)=(1+x)/(1-x^2)$ are both proper forms.

\medskip
Now we can sketch how the XinPF Algorithm computes $\CT_\lambda E$ in $K$ for $\lambda=x_{i_0}$.
In order to do so, we need to write $E$ in the following form.
\begin{align}
  \label{e-E-xform}
E= \frac{L(\lambda)}{\prod_{i=1}^n (1-u_i \lambda^{a_i})},  \qquad\text{ (not in proper form)}
\end{align}
where $L(\lambda)$ is a Laurent polynomial, $u_i$ are free of $\lambda$ and $a_i$ are positive integers for all $i$. The algorithm mainly relies on the following known results.
\begin{prop}\label{p-partialfraction}
Suppose the partial fraction decomposition of $E$ is given by
\begin{align}
  \label{e-E-parfrac}
E= P(\lambda)+ \frac{p(\lambda)}{\lambda^k} +\sum_{i=1}^n \frac{A_i(\lambda)}{1-u_i \lambda^{a_i}},
\end{align}
where the $u_i$'s are free of $\lambda$, $P(\lambda),p(\lambda),$ and the $ A_i(\lambda)$'s are all polynomials, $\deg p(\lambda)<k$, and
$\deg A_i(\lambda)<a_i$ for all $i$.
Then we have
$$\CT_\lambda E = P(0) + \sum_{u_i \lambda^{a_i} <1} A_i(0).$$
\end{prop}

The Proposition holds since, if written in proper form, we shall have
$$ \frac{A_i(\lambda)}{1-u_i \lambda^{a_i}} =\left\{
                   \begin{array}{ll}
                    \displaystyle \frac{A_i(\lambda)}{1-u_i \lambda^{a_i}} \ \mathop{\longrightarrow}\limits^{\CT_\lambda} A_i(0), & \text{ if } u_i \lambda^{a_i} <1; \\
                    \displaystyle \frac{A_i(\lambda)}{-u_i\lambda^{a_i} (1-\frac{1}{u_i \lambda^{a_i}})}=\frac{\lambda^{-a_i}A_i(\lambda)}{-u_i(1-\frac{1}{u_i \lambda^{a_i}})}
                    \ \mathop{\longrightarrow}\limits^{\CT_\lambda} 0, & \text{ if } u_i \lambda^{a_i}>1.
                   \end{array}
                 \right.
 $$

\begin{thm}\label{t-parfrac-single}
Let $E$ be as in \eqref{e-E-parfrac}. Then $A_s(\lambda)$ is uniquely characterized by
\begin{align}\label{e-As}
 \left\{
   \begin{array}{l}
      A_s(\lambda) \equiv E(1-u_s\lambda^{a_s})  \pmod{\langle 1-u_s\lambda^{a_s} \rangle} , \\
     \deg_\lambda A_s <a_s,   \end{array}
 \right.
 \end{align}
where  $\langle 1-u\lambda^a \rangle$ denotes the ideal generated by $1-u\lambda^a$.
\end{thm}
To compute $A_s(\lambda)$ explicitly by polynomial operations, we also need the following formula to invert $1-u_i \lambda^{a_i}$ for all $i\ne s$.
$$ \frac{1}{\lambda^{b} -v} \equiv \frac{1}{1-u^b v^{a}} \cdot \frac{1- (u \lambda^a)^b }{1-u \lambda^a}   \pmod{\langle 1-u \lambda^{a} \rangle}  \text{ if } u^bv^a \ne 1, $$
where we allow $v=0$ to handle the possible factor $\lambda^{-b}$ of $L$.

The above ideas have been carried out by the Maple package \texttt{Ell2}. See \cite{iterate} for further technical treatment. The package works fine for many practical problems but it may break down for some random problems. There are three bottlenecks in this algorithm.

\begin{enumerate}
 \item The computation of $P(\lambda)$ may be expensive since by the polynomial division algorithm we have to expand the denominator.

 \item There is no good way to deal with the case of non-coprime $1-u_i \lambda^{a_i}$ and $1-u_s \lambda^{a_s}$. This problem is called the multiple roots problem.

  \item The explicit formula of $A_s(\lambda)$ may have too many monomials. The number maybe as large as $(\text{the number of monomials in the numerator of } E) \times a_s^{n-1}$.
\end{enumerate}

The proposed Euclid style algorithm resolves these three bottlenecks. We outline the algorithm as follows.

\medskip
\noindent
CTEuclid Algorithm computes $\CT_\Lambda E$ in $K$ where $\Lambda=\{\lambda_1,\dots,\lambda_r\} \subseteq \{x_1,\dots, x_m\}$.

\begin{enumerate}
\item[1] Write $E$ in its proper form as in \eqref{e-properform}. Set
$$ O_0 = \frac{L}{(1-z_1M_1)(1-z_2M_2)\cdots (1-z_nM_n)},$$
and set $K'=K((z_n))\cdots ((z_1))$.

\item[2.1] Compute $O_{i+1}$ from $O_{i}$ for $i=0,1,\dots, r-1$ as follows.

     \item[2.2] For each summand $O_{ij}$ in $O_i$, choose a variable $\lambda$, compute $\CT_\lambda O_{ij}$ in $K'$ as in the next step, and collect the results into $O_{i+1}$.

     \item[2.3] Compute $\CT_{\lambda} O_{ij}$ in $K'$ by Theorem \ref{t-P(0)} and by Proposition \ref{p-recursion}.

\item[3] Dispel the slack variables $z_i$ for all $i$ from $O_r$ and give the output.
\end{enumerate}
In Step 1, we set the new working field $K'$ specified by $\texttt{vars}=[z_1,\dots, z_n, x_1,\dots, x_m]$.
By direct series expansion, it clearly holds that
      $$\CT_{\Lambda} E =\left( \CT_{\Lambda} O_0 \right)\Big|_{z_i=1,\; i=1,2,\dots, n}.$$
The adding of the slack variables $z_1,\dots, z_n$ avoids the multiple roots problem. Together with Step 3, we resolve the multiple roots bottleneck.

In Step 2, we are eliminating all the variables $\lambda_i$ to obtain $O_r= \CT_{\Lambda} O_0.$ Steps 2.1-2.2 are similar to that in the XinPF Algorithm.
The other two bottlenecks are resolved in Step 2.3, or in the next two subsections, where we avoid the explicit formula of $P(x)$  by Theorem \ref{t-P(0)}, and give a recursion for $A_s(x)$ by Proposition \ref{p-recursion}.

\subsection{A reduction to the contribution of a single factor \label{s-red-single}}
Now the major problem is to compute $\CT_\lambda E$ for $E$ as in \eqref{e-E-xform} in $K$.
Assume we had not introduced the slack variables in Step 1.
Our task is to compute
$$ \mathcal{A}_{ 1-u_s \lambda^{a_s}}\ E := A_s(0), $$
where $A_s(\lambda)$ is characterized by \eqref{e-As}. This definition will also be used for $1-u \lambda^a$ in the following two cases.

i) if
$1-u \lambda^a$ is coprime to the denominator of $E$ then $\mathcal{A}_{ 1-u\lambda^{a}}\ E =0$;

ii) if $u\lambda^a=u_s \lambda^{a_s}$ but $1-u \lambda^a$ is not coprime to $1-u_i \lambda^{a_i}$ for some $i\ne s$, then $A_s(\lambda)$ does not exist, and we
define $\mathcal{A}_{ 1-u\lambda^{a}}\ E $ to be ``does not apply".

In practice, we might be lucky enough to never meet case ii). Then $O_r$ is already the output. In the general situation or a complicated problem, we need to introduce
the   slack variables as in Step 1, so that case ii) will never appear in any computation of $\CT_\lambda O_{ij}$ in $K'$.

Thus in the new notation, Proposition \ref{p-partialfraction} reads
\begin{align}
  \label{e-E-ctxNote}
\CT_\lambda E = P(0)+ \sum_i \chi(u_i \lambda^{a_i}<1) \mathcal{A}_{ 1-u_i \lambda^{a_i}}\ E .
\end{align}

The following result can be used to avoid the computation of $P(0)$.
\begin{lem}\label{l-ctE}
Let $E$ be as in \eqref{e-E-xform}.
If $E$ is \emph{proper} in $\lambda$, i.e., the degree in the numerator is less than the degree in the denominator, then
\begin{align}
  \label{e-ctE}
\CT_\lambda E= \sum_{i=1}^n \chi(u_i\lambda^{a_i}<1) \mathcal{A}_{ 1-u_i x^{a_i}}\ E ;
\end{align}
If $E|_{\lambda=0}= \lim_{\lambda\to 0} E$ exists, then
\begin{align}\label{e-ctE-dual}
 \CT_\lambda E =E|_{\lambda=0}- \sum_{i=1}^n \chi(u_i\lambda^{a_i}>1) \mathcal{A}_{ 1-u_i x^{a_i}}\ E .\tag{\ref{e-ctE}$'$}
\end{align}
\end{lem}
Formula \eqref{e-ctE-dual} is a kind of dual of \eqref{e-ctE}.
Because of these two formulas, it is convenient to call the denominator factor $1-u_i\lambda^a_i$ \emph{contributing} if
$u_i\lambda^{a_i}$ is small and \emph{dually contributing} if $u_i \lambda^{a_i}$ is large.
Now we will also denote
$$ \CT_\lambda \frac{1}{\underline{1-u_s \lambda^{a_s}}} E (1-u_s \lambda^{a_s}) = \mathcal{A}_{ 1-u_s \lambda^{a_s}}\ E=A_s(0).$$
For this notation, one can think that when taking the constant term in $\lambda$, only the single underlined factor of the denominator contributes.

\begin{proof}[Proof of Lemma \ref{l-ctE}]
Suppose we have the partial fraction decomposition given in \eqref{e-E-parfrac}.
i) If $E$ is proper in $\lambda$ then $P(\lambda)=0$ and the first equality holds;
ii) If $E|_{\lambda=0}$ exists, then $p(\lambda)$ must be $0$. Now setting $\lambda=0$ and applying \eqref{e-E-ctxNote} gives
$$E|_{\lambda=0} = P(0)+\sum_{i} A_i(0) = \CT_\lambda E + \sum_{i=1}^n \chi(u_i\lambda^{a_i}>1) \mathcal{A}_{1-u_i \lambda^{a_i}}\ E.$$
The second equality then follows by subtraction.
\end{proof}

If $E$ is proper and has no pole at $\lambda=0$ then both formulas \eqref{e-ctE}
and \eqref{e-ctE-dual} apply and we can choose to use the simpler one.

\begin{thm}\label{t-P(0)}
  Let $E$ be as in \eqref{e-E-xform}. Split $L(\lambda)$ as $L_1(\lambda)+L_2(\lambda)$, where $L_1$ contains only positive powers in $\lambda$ and $L_2$ contains only nonpositive powers in $\lambda$. Set $E_i =E L_i(\lambda)/L(\lambda)$ for $i=1,2$. Then
\begin{align}
  \label{e-ctE12}
\CT_\lambda E = \sum_{i}\chi(u_i \lambda^{a_i}<1) (\mathcal{A}_{1-u_i \lambda^{a_i}}\ E_2)  - \sum_{i}\chi(u_i \lambda^{a_i}>1) (\mathcal{A}_{1-u_i \lambda^{a_i}}\ E_1).
\end{align}
\end{thm}
\begin{proof}
  By linearity and the fact that $E=E_1+E_2$, we have
  $$ \CT_\lambda E = \CT_\lambda E_1 +\CT_\lambda E_2.$$
It is clear that $E_1|_{\lambda=0}=0$ and that $E_2$ is proper. Now apply Lemma \ref{l-ctE} to get \eqref{e-ctE12}.
\end{proof}

\subsection{A recursion for the contribution of a single factor \label{s-rec-single}}

The contribution of a linear factor is easy:
\begin{align}
  \mathcal{A}_{1-u \lambda} \ E = \CT_\lambda \frac{1}{\underline{1-u\lambda}} E(1-u\lambda) = E(1-u\lambda)|_{\lambda=1/u}. \label{e-linearfactor}
\end{align}
However, effective computation for nonlinear factors was a long standing problem. See, e.g., \cite{MPAA2}. One can factor $1-u \lambda^a$ into linear factors using roots of unity, but there is no simple way to get rid of the roots of unity in the final outcome. We present here a Euclid style algorithm dealing with the nonlinear case.

Let $E$ be as in \eqref{e-E-xform} and let $1-u \lambda^a$ be a denominator factor of $E$. We construct as follows a nicer $E'$ with the property
$\mathcal{A}_{1-u \lambda^a} \ E = \mathcal{A}_{1-u \lambda^a} \ E'$.

For simplicity we assume $u\lambda^a=u_1\lambda^{a_1}$. Recall that
$$ A_1(\lambda) \equiv E(1-u\lambda^a)  \pmod{\langle 1-u\lambda^a \rangle}.$$

The following clearly holds,
$$\lambda^m \equiv u^{-\ell} \lambda^r\pmod  {\langle 1-u\lambda^a \rangle} \text{ if } m =\ell a +r .$$
Particularly, the \emph{remainder} $\mrem( \lambda^m, 1-u\lambda^a,\lambda) $ and the \emph{signed remainder} $\srem( \lambda^m, 1-u\lambda^a,\lambda)$ of $\lambda^m$ when dividing by $1-u\lambda^a$ is defined to be
\begin{align}
  \mrem( \lambda^m, 1-u\lambda^a,\lambda)&=u^{-\ell} \lambda^r , \text{ where } m =\ell a +r,\  0 \le r< a.\\
  \srem( \lambda^m, 1-u\lambda^a,\lambda)&=u^{-\ell} \lambda^r , \text{ where } m =\ell a +r,\  -a/2 < r\le a/2.
\end{align}
These definitions linearly extend for Laurent polynomials.

The new idea is that $A_1(0)$ can be unearthed by using a \emph{better} representative of $A_1(\lambda)+ \langle 1-u\lambda^a \rangle$ instead of the explicit formula of $A_1(\lambda)$. We have
\begin{align*}
  A_1(\lambda) &\equiv E (1-u_1 \lambda^{a_1}) \equiv \frac{L(\lambda)}{\prod_{i=2}^n (1-u_i \lambda^{a_i})}  \pmod{\langle 1-u\lambda^a \rangle}\\
  &\equiv \frac{L(\lambda)}{\prod_{i=2}^n (1-u_i\; \srem(\lambda^{a_i}, 1-u\lambda^a, \lambda) )}  \pmod{\langle 1-u\lambda^a \rangle}.
\end{align*}
This can be rewritten in the following form by the fact that $1/(1-v)=(-v)^{-1}/(1-v^{-1})$:
$$A_1(\lambda) \equiv \frac{\pm M L(\lambda)}{\prod_{i=2}^n (1-v_i \lambda^{b_i} )}  \pmod{\langle 1-u\lambda^a \rangle},$$
where $M$ is a monomial and $0\le b_i\le a/2$ for all $i$.

Now we have to split into two cases:

i) if all the $b_i$ are $0$, then $a$ divides all the $a_i$ and we immediately obtain
$$ A_1(\lambda) = \frac{\mrem (L(\lambda), 1-u \lambda^a, \lambda)}{\prod_{i=2}^n (1-u_i\; u^{-a_i/a} )}.$$
Setting $\lambda=0$ gives the desired $A_1(0)$.

ii) if at least one of $b_i$ is greater than $0$, then rewrite
$$ A_1(\lambda) \equiv \frac{ \lambda L'(\lambda)}{\prod_{i=2}^n (1-v_i \lambda^{b_i} )}  \pmod{\langle 1-u\lambda^a \rangle},$$
where
$$L'(\lambda)=\pm \mrem(\lambda^{-1} M L(\lambda), 1-u\lambda^a)$$ is a polynomial in $\lambda$ of degree less than $a$. Now comes
the crucial observation:
\begin{align*}
 A_1(0) = \CT_\lambda \frac{1}{\underline{1-u \lambda^a}}\; \frac{ \lambda L'(\lambda)}{\prod_{i=2}^n (1-v_i \lambda^{b_i} )}.
\end{align*}
In our notation, this is just
\begin{align*}
\mathcal{A}_{ 1-u \lambda^{a}} \ E =\mathcal{A}_{ 1-u \lambda^{a}}\ E',
\end{align*}
where
\begin{align} \label{e-E'single}
  E' = \frac{1}{{1-u \lambda^a}} \frac{ \lambda L'(\lambda)}{\prod_{i=2}^n (1-v_i \lambda^{b_i} )}
\end{align}
is a proper rational function with $E'|_{\lambda=0}=0$.
It follows by the partial fraction decomposition of $E'$ and then setting $\lambda=0$ that
\begin{align}
 \mathcal{A}_{ 1-u \lambda^{a}} \ E' = -\sum_{i=2}^n \mathcal{A}_{ 1-v_i \lambda^{b_i}} \ E'.
\end{align}
Note that the terms for each $b_i=0$ vanish.

Thus we have proved the following result.
\begin{prop}\label{p-recursion}
Let $E$ be as in \eqref{e-E-xform} and let $1-u \lambda^a=1-u_1 \lambda^{a_1}$ be a denominator factor.
If $a$ divides every $a_i, i\ge 2$, then
$$ \mathcal{A}_{ 1-u \lambda^{a}}\ E = \frac{\mrem (L(\lambda), 1-u \lambda^a, \lambda)}{\prod_{i=2}^n (1-u_i\; u^{-a_i/a} )} \Big|_{\lambda=0};$$
If at least one of the $a_i$ is not divisible by $a$, then construct $E'$ as in \eqref{e-E'single}. We have
$$\mathcal{A}_{ 1-u \lambda^{a}}\ E  = -\sum_{i}  \mathcal{A}_{ 1-v_i \lambda^{b_i}}\ E',$$
where $0<b_i\le a/2$ for all $i$ and the number of terms is at most $n-1$.
\end{prop}

Repeated application of Proposition \ref{p-recursion} will give a sum of simple rational functions. The number of terms only depends on the $a_i$'s and the process is similar to Euclid's gcd algorithm. Denote this number by $f(i; a_1,a_2,\dots,a_n)$ where $i$ corresponds to the factor $1-u_i \lambda^{a_i}$. Then $f(i; a_1,a_2,\dots,a_n)$ is recursively determined by the following rules:

\begin{enumerate}
  \item If $a_j=0$ for all $j\ne i$ then $f(i; a_1,a_2,\dots,a_n)=1$;

\item We have $f(i; a_1,a_2,\dots,a_n) =f(i; b_1,b_2 ,\dots,b_n)$, where $b_i=a_i$ and \\
$b_j=\min(\mrem(a_j,a_i), a_i- \mrem(a_j,a_i))$ for $j\ne i$.

\item If $a_j\le a_i/2$ for all $j\ne i$, then
$$ f(i; a_1,a_2,\dots,a_n) = \sum_{j\ne i} f(j, a_1,a_2,\dots,a_n).$$
\end{enumerate}

If $n=1$ then $f(1;a_1)=1$. If $n=2$ we also have $f(i;a_1,a_2)=1$, because the sum of the recursion contains a single term.
For $n=3$,  computational evidence suggests that $f(1;a_1,a_2,a_3)$ is almost $O((\log a_1)^2)$.
For larger $n$, we raise the following problem:

Let $f(i;a_1,a_2,\dots, a_n)$ be defined as above. Prove or disprove that $f(i;a_1,a_2,\dots, a_n)$ is a polynomial in $\log a_i$. Note that the obvious bound is
$f(i;a_1,a_2,\dots, a_n)< n^{\log_2 a_i}$.

If the answer is positive, then we will obtain a simple polynomial time algorithm at least for one variable elimination. But this might not be the right problem, since we have a lot of freedom to apply the partial fraction technique, and the current approach is too elementary.

\section{The Maple package \texttt{CTEuclid} \label{s-CTEuclid}}

The algorithm in Section 4 is implemented by the Maple package \texttt{CTEuclid}, which can be downloaded from the following link\\
 \centerline{https://www.dropbox.com/sh/scepodyyn4ff7ro/ffhqmeN7ne/MPA,}\\
where two demo files are provided to explain how to use the package. One file works on magic squares of order up to 5 and the other file works on the Sdd problem \cite{sdd5} of order up to 5. Both files contain the essential idea of the ``delay trick on slack variables" in Section \ref{s-sdd} for attacking the order 6 case. Here we only report the Ehrhart series for  magic squares of order 6.

\texttt{CTEuclid} is the first package designed for complicated or even benchmark problems.
It outperforms \texttt{Ell2} for high-dimensional problems, and is only a bit slower for some low-dimensional problems because
\texttt{Ell2} does not introduce slack variables.
%After loading the package, the command CTEuclid($E$,\texttt{vars}, \texttt{va}) will deliver $\CT_{\Lambda} E$ in $K$, where $E$ is an Elliott-rational function, $K$ is specified by the list \texttt{vars}$=[x_1,\dots, x_m]$ of all variables, and the \texttt{va}$=[\lambda_1,\dots, \lambda_r]$ stands for the set $\Lambda$.
For the sake of clarity, we explain by considering 
several knapsack-type problems below.

\subsection{Knapsack-type problems}
Let $a_0,a_1,\dots,a_n$ be positive integers with $a=(a_1,\dots, a_n)$, $\gcd(a_1,\dots,a_n)=1$ and $a_i\le a_0$ for all $i$, and let
$$ P= \{ x\in \mathbb{R}^n : ax=a_0, x\ge 0\}.$$
A basic problem is to determine if $P$ contains an integer vector, or how many integer vectors $P$ contains. The former is called the integer programming feasibility problem. See \cite{AardalLenstra} for an introduction to this topic. Here we concentrate on the second problem, which is also called the knapsack-type problem. Clearly we have
$$ \# P = [x^{a_0}]  \frac{1}{(1-x^{a_1}) \cdots (1-x^{a_n})} =  \CT_x \frac{1}{x^{a_0} (1-x^{a_1}) \cdots (1-x^{a_n})}.  $$

\begin{exa}
  Compute the following constant term:
  $$ \CT_x \frac{1}{x^{41} (1-x ) (1-x^5 ) (1-x^{14} )}.$$
\end{exa}
\begin{proof}
  [Solution.]
  We first add slack variables and get
 \begin{align*}
   \CT_x E =\CT_x \frac{x^{-41}}{ (1-x z_1) (1-x^5 z_2) (1-x^{14} z_3)}= \CT_x \frac{x^{-41}}{\underline{(1-x z_1) (1-x^5 z_2) (1-x^{14} z_3)}},
 \end{align*}
 where the three underlined factors are contributing. For the last factor we have
\begin{align*}
  \CT_x \frac{x^{-41}}{ (1-x z_1) (1-x^5 z_2) \underline{(1-x^{14} z_3)}} &= \CT_x \frac{x z_3^3}{(1-x z_1) (1-x^5 z_2) \underline{(1-x^{14} z_3)}} \\
&=- \CT_x \frac{x z_3^3}{\underline{(1-x z_1) (1-x^5 z_2)} (1-x^{14} z_3)},
\end{align*}
where in our notation, only the first two factors are contributing. The flexibility of our algorithm allows us to obtain the following combined form:
\begin{align*}
  \CT_x E &= \CT_x \frac{1}{x^{41} \underline{(1-x z_1) (1-x^5 z_2)} (1-x^{14} z_3)}+ \CT_x \frac{1}{x^{41} (1-x z_1) (1-x^5 z_2)\underline{ (1-x^{14} z_3)}} \\
   &= \CT_x \frac{x^{-41}}{ \underline{(1-x z_1) (1-x^5 z_2)} (1-x^{14} z_3)}- \CT_x \frac{x z_3^3}{ \underline{(1-x z_1) (1-x^5 z_2)} (1-x^{14} z_3)}\\
   &= \CT_x \frac{x^{-41}- x z_3^3}{\underline{(1-x z_1) (1-x^5 z_2)} (1-x^{14} z_3)}.
\end{align*}
Now the first contribution is simple:
\begin{align*}
 \CT_x \frac{x^{-41}- x z_3^3}{\underline{(1-x z_1)} (1-x^5 z_2) (1-x^{14} z_3)}= \frac{z_1^{41}- z_1^{-1} z_3^3}{ (1-z_1^{-5} z_2) (1-z_1^{-14} z_3)}.
\end{align*}
The contribution of the second factor becomes
\begin{align*}
 % \CT_x \frac{x^{-41}- x z_3^3}{(1-x z_1)\underline{ (1-x^5 z_2)} (1-x^{14} z_3)}
 \CT_x &\frac{x^{-41}- x z_3^3}{(1-x z_1)\underline{ (1-x^5 z_2)} (1-x^{-1} z_3 z_2^{-3})} \\
  &= \CT_x -\frac{x^{5} z_3^{-1} z_2^{12} - x^2 z_3^2 z_2^{3}}{(1-x z_1)\underline{ (1-x^5 z_2)} (1-x z_3^{-1} z_2^{3})} \\
  &= \CT_x \frac{x^{5}z_2^{12} z_3^{-1} - x^2 z_3^2 z_2^{3}}{\underline{ (1-x z_1)}(1-x^5 z_2) \underline{ (1-x z_3^{-1} z_2^{3})}} \\
  &=  \frac{z_1^{-5}z_2^{12} z_3^{-1} - z_1^{-2} z_3^2 z_2^{3}}{(1-z_1^{-5} z_2) { (1-z_1^{-1} z_3^{-1} z_2^{3})}}+ \frac{z_3^4z_2^{-3} - z_3^4z_2^{-3} }{{ (1-z_3z_2^{-3} z_1)}(1-z_3^5z_2^{-14}) }.
\end{align*}
Thus we obtain a sum of three terms and come to Step 3. We need to make a substitution so that
$ z_1^{-5} z_2,\ z_1^{-14} z_3,\ z_1^{-1} z_3^{-1} z_2^{3},\&\; z_3^5z_2^{-14}$ are not equal to $1$. One choice is $z_1=1, z_2=t, z_3=t$. Then the constant term becomes
\begin{align*}
  \frac{1-  t^3}{ (1-t)^2} + \frac{t^{11} -  t^5}{(1-t) { (1-t^{2})}}+ \frac{t - t }{{ (1-t^{-2} )}(1-t^{-9}) }=\frac{1-  t^3}{ (1-t)^2} + \frac{t^{11} -  t^5}{(1-t) { (1-t^{2})}}.
\end{align*}
Then we let $t=1+s$ and take the constant term in $s$ separately to get
\begin{align*}
 \CT_s &\frac{1- (1+s)^3}{ s^2} + \CT_s \frac{(1+s)^{11} -  (1+s)^5}{s^2(2+s) }\\
 &= -3 +[s^2] (1+11 s+55 s^2 -(1+5s +10 s^2)) \frac{1}{2} (1-s/2+s^2/4))\\
 &= -3 + [s] (6 + 45 s) (1/2-s/4) = -3 +\frac{45}{2}-\frac{6}{4} =18.
\end{align*}
\end{proof}

Next we consider a relatively complicated example, which is Example 1 of \cite{AardalLenstra}.
\begin{exa}
Show that the polytope $P$ contains no integer lattice points, where
  $$P=\{ x\in \mathbb{R}^3: 12,223 x_1+12,224 x_2 +36,671 x_3=149,389,505, x\ge 0\}.$$
\end{exa}
\begin{proof}[Sketch of the Proof.]
The problem is equivalent to computing the following constant term:
$$\CT_x {\frac {1}{{x}^{149389505} \left( 1-{x}^{12223} \right)  \left( 1-{x}^
{12224} \right)  \left( 1-{x}^{36671} \right) }}.
$$
\texttt{CTEuclid} will give a sum of 10 terms, which reduces by cancelation to a sum of 4 terms.
%\begin{multline*}
% {\frac {{z_{{1}}}^{12224}{z_{{2}}}^{12224}}{ \left( z_{{1}}{z_{{2}}}^{
%2}-z_{{3}} \right)  \left( -{z_{{2}}}^{12223}+{z_{{1}}}^{12224}
% \right) }}-{\frac {{z_{{2}}}^{24447}}{ \left( z_{{1}}{z_{{2}}}^{2}-z_
%{{3}} \right)  \left( -{z_{{2}}}^{12223}+{z_{{1}}}^{12224} \right) }} \\
%-
%{\frac {{z_{{2}}}^{48895}}{ \left( z_{{1}}{z_{{2}}}^{2}-z_{{3}}
% \right)  \left( {z_{{2}}}^{36671}-{z_{{3}}}^{12224} \right) }}+{
%\frac {{z_{{3}}}^{12224}{z_{{2}}}^{12224}}{ \left( z_{{1}}{z_{{2}}}^{2
%}-z_{{3}} \right)  \left( {z_{{2}}}^{36671}-{z_{{3}}}^{12224} \right)
%}}.
%\end{multline*}
By letting $z_1=1,z_2=t,z_3=t$ we obtain
$$-{\frac {{t}^{12223}}{ \left( t-1 \right)  \left( {t}^{12223}-1
 \right) }}+{\frac {{t}^{24446}}{ \left( t-1 \right)  \left( {t}^{
12223}-1 \right) }}-{\frac {{t}^{36670}}{ \left( t-1 \right)  \left( {
t}^{24447}-1 \right) }}+{\frac {{t}^{12223}}{ \left( t-1 \right)
 \left( {t}^{24447}-1 \right) }}
 .$$
To eliminate the slack variable $t=1$, we let $t=e^s$ and compute the constant term in $s$
for each term separately. For instance, the first term becomes,
\begin{align*}
  \CT_s -{\frac {{e}^{12223s}}{ \left( e^s-1 \right)  \left( {e}^{12223s}-1
 \right) }} &= [s^2] -{{e}^{12223s}}\times \frac {s}{ \left( e^s-1 \right)}\times \frac{s}{ \left( {e}^{12223s}-1
 \right) }\\
 &=[s^2] -(1+12223s+\frac12 12223^2s^2) (1-s/2+s^2/12) \\
 &\qquad \times (1/12223-s/2+12223s^2/12) =-{\frac {149365061}{146676}}.
\end{align*}
The four constant terms sum to $0$. This completes the proof.
\end{proof}

Still from the article \cite{AardalLenstra}, a very hard instance of the knapsack-type problem concerns the 4 dimensional polytope with
$$ a_0=89643481, (a_1,\dots, a_5)=( 12223 , 12224,36674,61119,85569).$$
Aardal and Lenstra show that $P$ contains no integer vectors in 0.01 second of cpu time while the Branch and Bound method takes more than
8139 seconds of cpu time. When dealing with this problem, \texttt{CTEuclid} gives 398 terms and returns $0$ in about 0.4 seconds of cpu time. The advantage of our algorithm is that we can compute the number $\# P$ for different $a_0$ in about the same time. For instance, if
$a_0=89643481 \times 1001$,  \texttt{CTEuclid}  still gives 398 terms and returns 94267024658624993843 in about 0.4 seconds.
It is worth noting that many of the 398 terms cancel
with only 118 terms left. It might be interesting to understand how these terms cancel with each other.
We also tried random examples with $100000 \le a_i \le  2500000$; The performance is not nice when $n\ge 5$. The algorithm does not seem to be of polynomial time.

\medskip
The above examples show that even if the final answer is simple, the middle step may give complicated results.
We make the following observation: Step 1 takes no time; Step 2 is the most important step, where we hope the number of terms $nt$ we get is small; Step 3 of dispelling the slack variables is the most time consuming step. Its running time is almost linear to $nt$.
This leads to the following two technical treatment when dealing with complicated or even benchmark problems.

\begin{enumerate}
  \item In Step 2, we save some data for later use: the number $nt$ and the data for every 1000 terms we obtained are saved in different files, and the data for all bad denominator factors are saved in a file.

      \item The running time for Step 3 is estimated as a linear function of $nt$. This information helps us to decide if we shall stop and try to use some tricks to reduce the number $nt$.
\end{enumerate}

\subsection{Direct computation of the Ehrhart series\label{s-Ehrhart}}
Given a bounded rational polytope  $P=\{\alpha : A\alpha =b,\; \alpha \ge 0 \} \subset \mathbb{R}^n$, the function
$$ i_P (k):= \# (k P \cap \mathbb{Z}^n) = \# \{\alpha \in \mathbb{Z}^n: A\alpha =k b,\; \alpha \ge 0 \}$$
defined for any positive integer $k$ was first studied by E. Ehrhart \cite{Ehrhart}. It is called the Ehrhart polynomial when the vertices of $P$ are integral and is called the Ehrhart quasi-polynomial for arbitrary rational polytopes \cite[Ch. 4]{EC1}. For us it is easier to describe it using generating functions.
The Ehrhart series of $P$ defined by
$$I_P(q)= \sum_{k\ge 0} i_P(k) q^k$$
is an Elliott-rational function. It has close connection with the Hilbert series of some graded algebras.

An important problem is to compute the Ehrhart quasi-polynomial for a given $P$. An earlier method is to compute $i_P(k)$ for sufficiently many $k$ and then use the Lagrange interpolation formula to construct $i_P(k)$. We can compute the Ehrhart series directly by the following constant term representation.
$$I_P(q) = \CT_{\Lambda} \frac{1}{
\prod_{j=1}^n (1- \lambda_1^{a_{1,j}} \lambda_2^{a_{2,j}} \cdots \lambda_r^{a_{r,j}} x_j)} \times \frac{1}{1-q \lambda_1^{-b_1} \cdots \lambda_r^{-b_r}} \Big|_{x_j=1}.$$
This corresponds to a rational cone or the homogeneous system $(A,-b) \alpha =0$. This leads to a combined method for Ehrhart series computation: Use \texttt{LattE} to do the rational cone decomposition and then use our way of eliminating the slack variables. There is no implementation for this approach yet.
We remark that a similar idea was proposed in \cite{LattE} to avoid the use of Brion's theorem. But the authors of \cite{LattE} only computed $i_P(k)$ for particular $k$.

Many benchmark problems are related to the computation of Ehrhart series.
The counting of magic squares and its variations is one of the common topics in both combinatorics and computational geometry. The definition  of magic squares is different in different literature. Here an $n$ by $n$ nonnegative integer matrix $M=(a_{i,j})_{n\times n}$ is said to be a magic square with magic sum $m$ if its row sums, column sums, and two diagonal sums are all equal to $m$. That is, the order $n$ magic square polytope $MS_n$ is defined by the following linear constraints:
\begin{align*}
  a_{i,1}+a_{i,2}+\cdots+a_{i,n}=1,   \text{ for }1\le i\le n \\
  a_{1,j}+a_{2,j}+\cdots+a_{n,j}=1,    \text{ for }1\le j\le n \\
  a_{1,1}+a_{2,2}+\cdots +a_{n,n}=1, \ a_{n,1}+a_{n-1,2}+\cdots +a_{1,n}=1.
\end{align*}

The determination of $I_{MS_n}(q)$ is known for $n=3$, and for $n=4$. Many algorithms meet trouble for the $n=5$ case. See e.g., \cite{MagicSquares}. The first solution for the order 5 magic squares was reported in \cite{Loera}.

Our approach is along the line of MacMahon's partition analysis. Let $\lambda_i$ index the $i$-th row equation for each $i$, let $\mu_j$ index the $j$-th column equation for each $j$, and let $\nu_1$ and $\nu_2$ index the two diagonal equations. Then it is not hard to see that
$$ \sum_{m=0}^\infty \sum_{M} \prod_{1\le i,j\le n} x_{i,j}^{a_{i,j}} q^m  = \CT_{\lambda, \mu,\nu} F_n(x, q;\lambda,\mu,\nu),$$
 where the second sum ranges over all magic squares $M$ with magic sum $m$, and
 $$ F_n(x, q;\lambda,\mu,\nu) = \prod_{1\le i,j\le n} \frac{1}{1-x_{i,j} \lambda_i\mu_j \nu_1^{\chi(i=j)} \nu_2^{\chi(i+j=n+1)}} \frac{1}{1- q(\lambda_1\cdots \lambda_n \mu_1\cdots\mu_n \nu_1\nu_2)^{-1}}.$$
In particular, setting $x_{i,j}=1$ gives the generating function for $I_{MS_n}(q)$.
$$ I_{MS_n}(q) = \CT_{\lambda, \mu,\nu}  \prod_{1\le i,j\le n} \frac{1}{1-\lambda_i\mu_j \nu_1^{\chi(i=j)} \nu_2^{\chi(i+j=n+1)}} \times \frac{1}{1- q(\lambda_1\cdots \lambda_n \mu_1\cdots\mu_n \nu_1\nu_2)^{-1}}.$$
We believe that the order 6 magic squares problem should be computed quickly by the proposed approach here. We describe our computation by the \texttt{CTEuclid} package in the next subsection.

\subsection{Computation for magic squares of order 6}
In Section \ref{s-Ehrhart} we have converted the Ehrhart series for the order $n$ magic squares polytope to a constant term.
Applying the \texttt{CTEuclid} package will give the desired Ehrhart series. There is no difficulty for the cases $n=3,4$. Indeed the author's \texttt{Ell2} package is faster in these two cases but meet memory problem for the $n=5$ case. Our \texttt{CTEuclid} package computes the $n=5$ case in about 2700 seconds of cpu time. The $n=6$ case is much more complicated, and was not known before. The  Ehrhart series for order 6 magic squares has been put at Sloane's integer sequence website \cite[A216039]{sloane}. It looks like
\begin{align*}
I_{MS_6}(q) &= \frac{(1-q)^3 N}{\left( 1-
{q}^{3} \right) ^{5}  \left(1- {q}^{4} \right) ^{5} \left( 1-{q}^{5}
 \right) ^{4} \left(1-{q}^{6} \right) ^{6} \left( 1-{q}^{7}
 \right) ^{3}    \left( 1-{q}^{8} \right) ^{2} \left( 1-{q}^{9} \right)
  \left(1- {q}^{10} \right)}\\
  &=  1+96q+14763q^2+957936q^3+33177456q^4+718506720q^5+\cdots
\end{align*}
where
 \begin{multline*}
   N={q}^{138}+99\,{q}^{137}+15057\,{q}^{136}+ \cdots\\
   + 21382798694422310755770332936 q^{69}
    +\cdots + 15057 q^2+99 q+1.
 \end{multline*}
The result is obtained by parallel computations modulo three different large primes. The total cpu time is about $70\times 3$ days. The author would like to thank his officemates for running these computations on their computers.

For the order 6 magic squares problem, Step 2 of the \texttt{CTEuclid} algorithm takes about 10 hours cpu time to obtain $nt$ terms. Since $nt$ is too large, we have to split the data and save it in different files, with each file containing 1000 terms. The run time for Step 3 can be estimated for it is about linear in $nt$. Because $nt$ is large, we must do the computation modulo a large prime to avoid the large integer problem. Here we choose $p_1=636,286,597$ for our first computation. Our estimated run time for Step 3 is about 108 days of cpu time. This is based on the observation that it takes about 5 minutes to dispel the slack variables for the $1000$ terms in each file.

Compared with 108 days, it is a small pay off to try to reduce the number $nt$. The flexibility of our algorithm allows us to reduce $nt$ by one third so that Step 3 can be done in about 70 days. The idea is that we can delay the adding of slack variables, which will be explained in Section \ref{s-sdd}. The saved data can be reused in Step 3 by computing it modulo different primes. Indeed we also did the computation modulo $p_2=460,710,223,302,903,961$ and $p_3=1,073,129,417,747,493,923$.

Finally we use the Chinese remainder theorem to reconstruct the generating function $N/D$. We conclude that this is the desired solution because
the maximum coefficient in $N$ is about $6.797227759\times 10^{-17} p_1p_2p_3$. Of course, for a rigorous proof, one needs a bit more work. A possible approach is
based on the following observations: If $kP$ is a $d$ dimensional integral simplex, then $I_P(q)=N'/(1-q^k)^{d+1}$ for some polynomial $N'$ with nonnegative integer coefficients
(see e.g., \cite[Ch. 4]{EC1}); The above statement still holds if $kP$ is a $d$ dimensional integral polytope by using a carefully chosen unsigned simplex decomposition; $N'|_{q=1}/k^{d+1}$ is well-known to be the relative volume of $P$, which can be estimated by known methods in Geometry.

\subsection{The Sdd\emph{k}
problem and the flexibility of our algorithm \label{s-sdd}}
The following constant term is known as the Sdd$k$ problem in \cite{sdd5}, which has connections with symmetric functions, representation theory and invariant theory. The term Sdd stands for the Schur function $s_{d,d}$ indexed by the partition $(d,d)$.
\begin{align}
 W_k(q )=\CT_{a_1,\dots, a_k} \frac{{ \prod_{i=1}^k}\big(1-{ a_i^2}\big)}{
{\displaystyle \prod_{S\subseteq \{1,\dots, k\}}}\Big(1-q \prod _{i\in S }
a_i/{\prod _{j\not\in S }}a_j\Big)}
\label{I.13}
\end{align}
The Sdd$k$ series for $k\le 4$ are nice:
\begin{align}
 W_2( {q})= \frac{1}{1-q^2}\qquad W_3(q)= \frac{1}{
1-q^4} \qquad  W_4(q)= \frac{1}{(1-q^2)(1-q^4)^2(1-q^6)},
\label{3.20}
\end{align}
but the Sdd5 series $W_5(q)$ has a large numerator.
The series $ W_5(q)$ was first obtained by Luque-Thibon \cite{5} in the context of quantum computing.
Their computation was carried out by a brute force use of the partial
fraction algorithm of the author.

The Sdd5 problem is the problem on which \texttt{LattE} fails but \texttt{Ell2} succeeds. The corresponding polytope for the Sdd5 problem is of dimension $27=2^5-5$. It is the intersection of 5 hyperplanes but has 2712 vertices, which make it expensive to apply Brion's theorem.

With the help of group actions, we were able to use \texttt{Ell2} to solve the Sdd5 problem in 5 minutes \cite{sdd5}. Now we can recompute $W_5(q)$ directly using the \texttt{CTEuclid} package in about 45 minutes. But with a few pre-works, we can recompute $W_5(q)$ in about 3 minutes.

Here we illustrate the flexibility of our algorithm by computing $W_k(q)$ for $k=3$ by hand. In this case two pre-works already compute $W_3(q)$. The idea extends for larger $k$ but then we need the help of the \texttt{CTEuclid} package.

Our first simplification is by working in the ring $\mathbb{Q}[a_1^{\pm 1},\dots,a_k^{\pm 1}][[q]]$, which has a simple but useful property that any invertible changing of variables by $a_1\to M_1, \dots,$ $a_k\to M_k$, where the $M_i$ are monomials in the $a$'s, does not change the constant term.
Let
$$ F= \frac{\left( 1-{a_{{1}}}^{-2} \right)  \left( 1-{a_{{2}}}^{-2} \right)
}{  \left( 1-{\frac {qa_{{1}}a_{{3}}}{a_{{2}}}}
 \right)  \left( 1-{\frac {qa_{{1}}}{a_{{2}}a_{{3}}}} \right)  \left(
1-{\frac {qa_{{2}}a_{{3}}}{a_{{1}}}} \right)  \left( 1-{\frac {qa_{{2}
}}{a_{{1}}a_{{3}}}} \right)\left( 1-{\frac {qa_{{3}}}{a_{{1}}a_{{2}}
}} \right)  \left( 1-{\frac {q}{a_{{1}}a_{{2}}a_{{3}}}} \right)
}, $$
which is invariant under the substitution of $a_3$ by $a_3^{-1}$. Then we have
\begin{align*}
  W_3(q)&= \CT_{a_1,a_2,a_3} F \cdot \frac{1-a_3^{-2} }{\left( 1-qa_{{1}}a_{{2}}a_{{3}} \right)  \left( 1-{\frac {qa_{{1}}a_{
{2}}}{a_{{3}}}} \right) } \\
(\text{ by partial fraction in } q)&= \CT_{a_1,a_2,a_3} F\cdot \frac{1}{ 1-{q}{a_1a_2a_3}}- \CT_{a_1,a_2,a_3} F \cdot   \frac{a_3^{-2}}{ 1-\frac {qa_{{1}}a_{{2}}}{a_{{3}}
}}
 \\
(\text{ by } a_3 \to a_3^{-1} \text{ in the first term}) &=  \CT_{a_1,a_2,a_3} F \cdot \frac{1}{ {1-\frac {q a_1 a_2}{a_3}}}
 - \CT_{a_1,a_2,a_3} F \cdot \frac{a_3^{-2}}{ {1-\frac {q a_1 a_2}{a_3}}}
 \\
&= \CT_{a_1,a_2,a_3} F \cdot \frac{1-a_3^{-2}}{ {1-\frac {q a_1 a_2}{a_3}}}.
\end{align*}

Next we make the change of variables by $a_k\to a_1a_2\cdots a_k$, followed by the change of variables by $a_i\to a_i^{1/2}$ for $i\le k-1$. Then $W_3(q)$ becomes the constant term of the following rational function.
$${\frac {\left( 1-a_{{1}}^{-1} \right)  \left(1- a_{{2}}^{-1}
 \right)  \left(1- (a_{{1}}a_{{2}}{a_{{3}}}^{2})^{-1} \right) }{ \left(1-\frac{q}{a_3} \right)  \underline{\left( 1-qa_{{1}}a_{{3}} \right)}  \left( 1-\frac{q}{a_2 a_3}\right)  \left( 1-qa_{{2}}a_{{3}} \right)  \left(1 - \frac{q}{a_1a_3} \right)  \left( 1-qa_{{3}} \right)  \left( 1-\frac{q}{a_1a_2a_3} \right) }}.$$

The second pre-work is based on the following observation: the introduction of the slack variables in Step 1 is intended to solve the multiple roots problem, so it should be delayed when suitable. We try to set $O_0=E$ under the larger working field $K=\mathbb{Q}((a_1))\cdots ((a_k))((q))$ in Step 1 and see if Step 2 works for some variable $\lambda$. If this succeeds, we put the result in $O_1$ and try this method for each summand of $O_1$, and so on. This delay trick has been used in the two demo files, as reported at the beginning of Section \ref{s-CTEuclid}.

For the sake of clarity, we illustrate the second pre-work in detail for the computation of $W_3(q)$ by hand. 
We will compute $W_3(q)$ only by Lemma \ref{l-ctE} and Equation \eqref{e-linearfactor} for linear factors. First take the constant term in $a_1$, where we have underlined the only contributing (linear) factor. This is done by removing the underlined factor and then setting $a_1=1/(q a_3)$. Two factors in the numerator cancel with the factors in the denominator. We obtain:
$$\frac{ \left(1- a_{2}^{-1}
 \right)   }{ \left(1-\frac{q}{a_3} \right)  \underline{\left( 1-qa_{{2}}a_{{3}} \right)}  \left(1 - q^2 \right)   \left( 1-\frac{q^2}{a_2} \right) }.$$
Note that this cancelation reduces the dimension of the problem, but it will not happen if we add the slack variables at the beginning. Taking the constant term similarly in $a_2$ gives
\begin{align*}
W_3(q) &= \CT_{a_3}   \frac{ \left(1- qa_3
 \right)   }{\underbrace{ \left(1-\frac{q}{a_3} \right)}   \left(1 - q^2 \right)   \left( 1-{q^3a_3} \right) } =   \frac{ \left(1- q^2
 \right)   }{  \left(1 - q^2 \right)   \left( 1-{q^4} \right) } =\frac{1}{1-q^4},
\end{align*}
where we need to be careful when taking the constant term in $a_3$: the rational function is not proper. We have under-braced the only dually contributing linear factor and used the dual formula of Lemma \ref{l-ctE}.

For the Sdd5 computation, the second pre-work proceeds to eliminate 3 variables to obtain
$$ W_5(q) = \sum_{i=1}^{62} \CT_{a_1,a_2}  T_i,$$
where $T_i$ are simple rational functions looking like:
\begin{multline*}
  \frac {2{a_{{2}}}^{7}{a_{{1}}}^{6}{q}^{3}}{ \left( -a_{{2}}+q
 \right) ^{3} \left( -a_{{1}}a_{{2}}+q \right) ^{3} \left( a_{{1}}-1
 \right) ^{2} \left( {q}^{2}a_{{1}}-1 \right) ^{2} \left( -a_{{2}}+{q}
^{3} \right) ^{2} \left( -1+{q}^{2} \right) ^{3}}\\
\times\frac1{ \left( -a_{{1}}a_{{2}
}+{q}^{3} \right) ^{2} \left( -1+qa_{{1}}a_{{2}} \right) ^{2} \left( -
1+qa_{{2}} \right) ^{2} \left( -a_{{1}}+{q}^{2} \right) ^{2}}
\end{multline*}
Note that most of the $T_i$ have monomial numerators due to some cancelations. Now we can use the \texttt{CTEuclid} package for each $T_i$ separately. It is crucial to work in the field $K$ of iterated Laurent series so that we can write each $T_i$ in its proper form and add the slack variables.

In this way we can reconstruct $W_5(q)$ in only about 3 minutes of cpu time. These ideas allow us to construct $W_6(q)$, which was first obtained (without proof) by geometric methods by Kraus and Wallach \cite{sdd6}.
\section{Concluding remark}

For the core problem in MacMahon's partition analysis described in Problem \ref{prob-heart}, we have developed two very different algorithms:
Algorithm 1 is a polynomial time algorithm in theory in Section 3 based on Barvinok's polynomial time algorithm; Algorithm 2 is an elementary Euclid style algorithm with implementation \texttt{CTEuclid} in Section 4, along the line of MacMahon's partition analysis. Both algorithms use the subalgorithm for dispelling the slack variables, which extends Barvinok's idea to the multivariable specialization.

Algorithm 1 is polynomial but has two weaknesses. i) The use of Brion's theorem may be costly if the number of vertices of the corresponding polytope $P$ is large. ii) It can not deal with polynomial numerators ``uniformly". This practical issue has been addressed in Remark \ref{rem-slack2}. There is no rigorous definition of ``uniform". What we mean here is to avoid trivial splitting of the numerator into monomials, since the resulting subproblems have the same complexity of the original one.

Algorithm 2 is not polynomial, but has some advantages. i) It can deal with polynomial numerator ``uniformly". The use of Theorem \ref{t-P(0)} and Proposition \ref{p-recursion} allows us to eliminate one variable and split into some subproblems. This is not a trivial splitting since the subproblems have fewer variables to eliminate. ii) It has a lot of flexibility in the framework of iterated Laurent series. This is especially true because we can use the delay trick on slack variables as explained in Section \ref{s-sdd}.

In practice, Algorithm 2 performs well when the entries of $A_{r\times n}$ are small and $r$ is much smaller than $n$. We do not have an implementation of Algorithm 1 yet.

Both algorithms are applicable to the general core problem and are designed for complicated or even benchmark problems.
In a complicated practical problem like order 6 magic squares counting, the last step of dispelling the slack variables takes more than 99 percent of the run time. Thus we shall consider trying to use the delay trick on slack variables to have nontrivial splitting into subproblems. We only say ``try" here because there is no guarantee that such splitting must give fewer terms to improve the performance.

The flexibility of our framework of iterate Laurent series makes it possible to improve on the CTEuclid Algorithm. We do not know how to improve on the geometric side. Step 2 is the crucial step, and we shall concentrate on reducing the number of terms obtained in this step so that the running time in Step 3 will be significantly reduced.

There are many ideas to improve the algorithm. An interaction with known theories will give hints for improvements. For instance, Stanley's monster reciprocity theory contains some algorithmic ideas. See \cite{StanMonster,XinMonster}. We outline below a possible improvement by noticing that
the number of terms obtained by Proposition \ref{p-recursion} is dependent on the entries $a_i$.

Let us consider the linear Diophantine system $A\alpha =b$ with augmented matrix
$ (A, b).$ Clearly elementary row operation will not change the solution set. So it is possible to find a matrix $(A',b')$ with small entries, and with the same solution set. This step may be achieved by the well-known Lenstra Lenstra Lovasz's (LLL) basis reduction algorithm \cite{LLL1,LLL2}.
The author is considering upgrading the \texttt{CTEuclid} package by using this idea.

Our ultimate goal is to develop a  classic algorithm in this subject in the near future. We believe that such an algorithm should contain the following features.
\begin{enumerate}
\item We shall deal with the inhomogeneous case directly and avoid using Brion's theorem, which is too expensive when the number of vertices is large.

\item We shall give a decomposition dealing with Laurent polynomial numerators in a uniform way. The outcome will be analogous to simplicial cones.

\item We shall apply Barvinok's decomposition of simplicial cones into unimodular cones or the like, which we believe to be key idea of Barvinok's polynomial algorithm.
\end{enumerate}

\medskip
\noindent
{\small \textbf{Acknowledgements:} The author would like to thank the anonymous referees for valuable suggestions and insightful questions to improve the presentation. This work was partially supported by the Natural Science Foundation of China (11171231).}

\end{document}